\documentclass[%
 aip,
 jmp,%
 amsmath,amssymb,
 reprint,%
]{revtex4-1}
\usepackage{amsfonts}
\usepackage{tikz}
\usetikzlibrary{arrows,automata}
\usepackage[latin1]{inputenc}
\usepackage{verbatim}
\usepackage{amsmath}
\usepackage{mathptmx}
\usepackage{graphicx}
\usepackage{dcolumn}
\usepackage{bm}
\newtheorem{theorem}{\textbf{Theorem}}

\newtheorem{definition}{\textbf{Definition}}

\newtheorem{proof}{\textbf{Proof}}

\begin{document}

\preprint{AIP/123-QED}

\title[]{ENERGY EVOLUTION OF MULTI-SYMPLECTIC METHODS FOR MAXWELL EQUATIONS
WITH PERFECTLY MATCHED LAYER BOUNDARY\footnote{Error!}}
\thanks{Authors are supported by NNSFC (NO. 91130003,
NO. 11021101 and NO. 11290142).}

\author{Jialin Hong}
 \email{hjl@lsec.cc.ac.cn.}
 
\author{Lihai Ji}
 \email{jilihai@lsec.cc.ac.cn.}
\affiliation{%
State Key Laboratory of Scientific and Engineering Computing,
Institute of Computational Mathematics and Scientific/Engineering
Computing, Academy of Mathematics and Systems Science, Chinese
Academy of Sciences, 100080 Beijing, People's Republic of China
}%

\date{\today}

\begin{abstract}
In this paper, we consider the energy evolution of multi-symplectic
methods for three-dimensional (3D) Maxwell equations with perfectly
matched layer boundary, and present the energy evolution laws of
Maxwell equations under the discretization of multi-symplectic Yee
method and general multi-symplectic Runge-Kutta methods.
\end{abstract}

\maketitle

\section{Introduction}
\label{intro}Maxwell equations are the most foundational equations
in electromagnetism and are widely applied to many application
fields, such as aeronautics, electronics, and biology
\cite{RefA,RefB}, etc. They are mathematical expressions of the
natural laws correlative fields, such as Amp\`{e}re's law and
Faraday's law \cite{RefC}. On the other hand, in lossless medium,
the electromagnetic energy of the wave is constant at different
times \cite{RefI}. As we all know, to preserve the energy is greatly
important in constructing numerical schemes for different physical
problems. However, numerical methods, with some boundary conditions,
can not preserve the energy exactly in general cases. Therefore, it
is important and necessary to investigate the energy evolution of
Maxwell equations under numerical discretization with some boundary
conditions. The purpose of this paper is to study the energy
evolution of multi-symplectic methods for 3D Maxwell's equations
with perfectly matched layer (PML) boundary.

It has been recognized that symplectic structure-preserving
numerical methods have significant superiority than non-symplectic
methods in numerical solving Hamiltonian ordinary differential
equations (ODEs) and Hamiltonian partial differential equations
(PDEs) \cite{RefD}. At the end of last century, symplectic
integrators have been generalized to multi-symplectic ones
\cite{RefI,RefJ,RefE,RefF,RefG,RefK,RefH}. And multi-symplectic
integrators have been applied to Maxwell equations. For examples,
\cite{RefH} discussed the self-adjointness of the Maxwell equations
with variable coefficients $\epsilon$ and $\mu$, and showed that the
equations have the multi-symplectic structure. \cite{RefI} firstly
proposed an unconditionally stable, energy-conserved, and
computational efficiently scheme for two-dimensional (2D) Maxwell
equations with an isotropic and lossless medium. The further
analysis in the case of 3D was studied in \cite{RefJ}. Meanwhile,
\cite{RefK} proposed a kind of splitting multi-symplectic
integrators method for Maxwell equations in three dimensions, which
was proved to be unconditionally stable, non-dissipative, and of
first order accuracy in time and second order accuracy in space.

It is well known that the PML boundary conditions are widely applied
to the numerical simulation Maxwell equations. In 1993, Berenger
\cite{RefP,RefQ} firstly proposed the PML technique, which is based
on modifying the PDEs away from all physical boundaries such that
absorbing outgoing waves from the computation domain. It is a simple
and straightforward technique, easily implemented for both two and
three space dimensions using either cartesian or cylindrical
coordinates. However, to the best of our knowledge, the
investigation of multi-symplectic methods for Maxwell equations with
PML boundary does not exist. In this paper, inspired by this
problem, we investigate the energy evolution of general
multi-symplectic methods for Maxwell equations with Berenger's PML
boundary.

The rest of this paper is organized as follows. In Section 2, we
begin with some preliminary results about 3D Maxwell equations and
Berenger's PML systems. An equivalent formulation to Berenger's PML
systems is introduced in Section 3. In Section 4, we present the
energy evolution laws of multi-symplectic Yee method and general
multi-symplectc Runge-Kutta methods for 3D Maxwell equations with
PML boundary.
\section{Preliminary results}
\label{sec:1} Notations. We denote by $(\cdot , \cdot)$ the $L^{2}$
scalar product, $\|\cdot\|_{H^{s}}$ the norm in $H^{s}$.
\subsection{3D Maxwell equations}
\label{sec:2} For a linear homogeneous medium within linear
isotropic material with the permittivity $\varepsilon$ and the
permeability $\mu$, the scattering of electromagnetic waves without
the charges or the currents can be described by the 3D Maxwell
equations in curl formulation
\begin{eqnarray}
\frac{\partial \textbf{E}}{\partial
t}&=&\frac{1}{\varepsilon}\nabla\times\textbf{H}, \label{2.1.1}\\
\frac{\partial \textbf{H}}{\partial
t}&=&-\frac{1}{\mu}\nabla\times\textbf{E}, \label{2.1.2}
\end{eqnarray}
where $\textbf{E}=(E_{x}, E_{y}, E_{z})^{T}$ and $\textbf{H}=(H_{x},
H_{y}, H_{z})^{T}$ represent the electric field and the magnetic
field, respectively. The domain $\Omega\times[0, T]=[0, a]\times[0,
b]\times[0, c]\times[0, T]$ under consideration is occupied by this
medium and surrounded by perfect conductors.

The curl equations (\ref{2.1.1}) and (\ref{2.1.2}) can be written as the
componentwise formula
\begin{equation}\label{2.1.3}
 \frac{\partial}{\partial t}
\left[\begin{array}{ccccccc}
E_{x}\\[1mm]
E_{y}\\[1mm]
E_{z}\\[1mm]
H_{x}\\[1mm]
H_{y}\\[1mm]
H_{z}\\[1mm]
\end{array}\right]=\left[\begin{array}{ccccccc}
\frac{1}{\varepsilon}(\frac{\partial H_{z}}{\partial y}-\frac{\partial H_{y}}{\partial z})\\[3mm]
\frac{1}{\varepsilon}(\frac{\partial H_{x}}{\partial z}-\frac{\partial H_{z}}{\partial x})\\[3mm]
\frac{1}{\varepsilon}(\frac{\partial H_{y}}{\partial x}-\frac{\partial H_{x}}{\partial y})\\[3mm]
\frac{1}{\mu}(\frac{\partial E_{y}}{\partial z}-\frac{\partial E_{z}}{\partial y})\\[3mm]
\frac{1}{\mu}(\frac{\partial E_{z}}{\partial x}-\frac{\partial E_{x}}{\partial z})\\[3mm]
\frac{1}{\mu}(\frac{\partial E_{x}}{\partial y}-\frac{\partial E_{y}}{\partial x})\\[3mm]
\end{array}\right].
\end{equation}

When the medium is lossless, then by Green's formula it gets the
following invariants:
\begin{eqnarray}
Energy ~~&I&: \int_{\Omega}(\varepsilon|\textbf{E}(\textbf{x},t)|^{2}+\mu|\textbf{H}(\textbf{x},t)|^{2})d\Omega=Constant,\nonumber\\
Energy ~~&II&: \int_{\Omega}(\varepsilon|\frac{\partial \textbf{E}(\textbf{x},t)}{\partial t}|^{2}+\mu|\frac{\partial \textbf{H}(\textbf{x},t)}{\partial t}|^{2})d\Omega=Constant. \nonumber
\end{eqnarray}

The first invariant is called Poynting theorem in
electromagnetism and it can be easily verified, and the second is a
little more complex. For more details, see \cite{RefI}.

In the 2D transverse electric (TE) polarization case, the electric
and magnetic field read $\textbf{E}=(E_{x}, E_{y}, 0)^{T}$,
$\textbf{H}=(0, 0, H_{z})^{T}$. Therefore, the Maxwell equations
(\ref{2.1.1}) and (\ref{2.1.2}) become
\begin{eqnarray}\label{2.1.6}
\left\{ \begin{array}{cccc}
\frac{\partial E_{x}}{\partial t}&=&\frac{\partial H_{z}}{\partial
y},\\[2mm]
\frac{\partial E_{y}}{\partial t}&=&-\frac{\partial H_{z}}{\partial
x},\\[2mm]
\frac{\partial H_{z}}{\partial t}&=&\frac{\partial E_{x}}{\partial
y}-\frac{\partial E_{y}}{\partial x}.
\end{array}\right.
\end{eqnarray}

Let $Z=(H_{x},H_{y},H_{z},E_{x},E_{y},E_{z})^{T}$. Then the
componentwise formula (\ref{2.1.3}) is  multi-symplectic, i.e.,
\begin{equation}\label{2.1.7}
MZ_{t}+K_{1}Z_{x}+K_{2}Z_{y}+K_{3}Z_{z}=0,
\end{equation}
where
\begin{small}
$$M=\left(\begin{array}{ccccccc}
0&-I_{3\times3}\\[1mm]
I_{3\times3}&0\\[1mm]
\end{array}\right),
K_{p}=\left(\begin{array}{cccccccc}
\varepsilon^{-1}R_{p}&0\\[1mm]
0&\mu^{-1}R_{p}\\[1mm]
\end{array}
\right),~~\forall p=1, 2, 3.$$
\end{small}

The sub-matrix $I_{3\times3}$ is a
$3\times3$ identity matrix and
\begin{small}
$$
R_{1}=\left(\begin{array}{ccccccc}
0&0&0\\[1mm]
0&0&-1\\[1mm]
0&1&0\\[1mm]
\end{array}\right),
R_{2}=\left(\begin{array}{ccccccc}
0&0&1\\[1mm]
0&0&0\\[1mm]
-1&0&0\\[1mm]
\end{array}\right),
R_{3}=\left(\begin{array}{ccccccc}
0&-1&0\\[1mm]
1&0&0\\[1mm]
0&0&0\\[1mm]
\end{array}\right).$$
\end{small}

The multi-symplectic formulation in (\ref{2.1.7}) preserves the following
multi-symplectic structure
\begin{equation}\label{2.1.8}
\frac{\partial}{\partial t}dZ\wedge MdZ+\sum^{3}_{p=1}\frac{\partial}{\partial x_{p}}dZ\wedge K_{p}dZ=0,
\end{equation}
where $dZ$ is the solution of the variational equation
associated with (\ref{2.1.7}).

Let $[\partial_{t}]^{n}_{i,j,k}$ and
$[\partial_{x_{p}}]^{n}_{i,j,k}$ denote the discretization of
$\frac{\partial}{\partial t}$ and $\frac{\partial}{\partial x_{p}}$
(for $p=1,2,3$), where $n$ is the temporal index and $i,j,k$ are the
spatial indices in the discrete system. For the multi-symplectic
Hamiltonian PDEs (\ref{2.1.7}), we consider the following numerical
discretization
\begin{equation}\label{2.1.9}
M[\partial_{t}]^{n}_{i,j,k}Z^{n}_{i,j,k}+\sum^{3}_{p=1}K_{p}[\partial_{x_{p}}]_{i,j,k}^{n}Z^{n}_{i,j,k}=0.
\end{equation}

\begin{definition}
The numerical method (\ref{2.1.9}) is
multi-symplectic if it preserves the discrete version of the
multi-symplectic structure in (\ref{2.1.8}). That is,
\begin{small}
\begin{equation}\label{2.1.10}
[\partial_{t}]^{n}_{i,j,k}\Big(dZ^{n}_{i,j,k}\wedge MdZ^{n}_{i,j,k}\Big)+\sum^{3}_{p=1}[\partial_{x_{p}}]^{n}_{i,j,k}\Big(dZ^{n}_{i,j,k}\wedge K_{p}dZ^{n}_{i,j,k}\Big)=0. \end{equation}
\end{small}
\end{definition}

From the multi-symplectic Hamiltonian formulation given by (\ref{2.1.7}),
of which its solution preserves the multi-symplectic structure
(\ref{2.1.8}). Now, we list several multi-symplectic numerical schemes
applied to Maxwell equations given in
this section.\\

 $\bullet$ Yee method

This method is the basis of the highly popular numerical methods
known as the FDTD methods \cite{RefC}. Yee method is constructed by
central difference in both space and time based on a half-step
staggered grid. It is a second-order method and is conditionally
stable.
Recently, \cite{RefS} showed that Yee method is multi-symplectic by the discrete variational principle, so we
call it the multi-symplectic Yee method.\\

 $\bullet$ Multi-symplectic Runge-Kutta methods

Applying the symplectic Runge-Kutta methods in both time and space
to Maxwell equations (\ref{2.1.7}) leads to the multi-symplectic
Runge-Kutta methods. \cite{RefF} presented the sufficient conditions
of multi-symplecticity Runge-Kutta methods for Hamiltonian PDEs.

\subsection{Berenger's PML system for 3D Maxwell equations}
\label{sec:3}In the PML medium, each component of electromagnetic
field is split into two parts. In cartesian coordinates, the six
components yield 12 subcomponents denoted by $E_{xy}$, $E_{xz}$,
$E_{yx}$, $E_{yz}$, $E_{zx}$, $E_{zy}$, $H_{xy}$, $H_{xz}$,
$H_{yx}$, $H_{yz}$, $H_{zx}$, $H_{zy}$, and the Maxwell equations
read,
\begin{eqnarray}
\varepsilon\frac{\partial E_{xy}}{\partial
t}+\sigma_{y}E_{xy}&=&\frac{\partial (H_{zx}+H_{zy})}{\partial y},\label{2.2.a}\\
\varepsilon\frac{\partial E_{xz}}{\partial
t}+\sigma_{z}E_{xz}&=&-\frac{\partial (H_{yz}+H_{yx})}{\partial z},\label{2.2.b}\\
\varepsilon\frac{\partial E_{yz}}{\partial
t}+\sigma_{z}E_{yz}&=&\frac{\partial (H_{xy}+H_{xz})}{\partial z},\label{2.2.c}\\
\varepsilon\frac{\partial E_{yx}}{\partial
t}+\sigma_{x}E_{yx}&=&-\frac{\partial (H_{zx}+H_{zy})}{\partial x},\label{2.2.d}\\
\varepsilon\frac{\partial E_{zx}}{\partial
t}+\sigma_{x}E_{zx}&=&\frac{\partial (H_{yz}+H_{yx})}{\partial x},\label{2.2.e}\\
\varepsilon\frac{\partial E_{zy}}{\partial
t}+\sigma_{y}E_{zy}&=&-\frac{\partial (H_{xy}+H_{xz})}{\partial y},\label{2.2.f}\\
\mu\frac{\partial H_{xy}}{\partial
t}+\sigma_{y}^{\ast}H_{xy}&=&-\frac{\partial (E_{zx}+E_{zy})}{\partial
y}, \label{2.2.g}\\
\mu\frac{\partial H_{xz}}{\partial
t}+\sigma_{z}^{\ast}H_{xz}&=&\frac{\partial (E_{yz}+E_{yx})}{\partial
z}, \label{2.2.h}\\
\mu\frac{\partial H_{yz}}{\partial
t}+\sigma_{z}^{\ast}H_{yz}&=&-\frac{\partial (E_{xy}+E_{xz})}{\partial
z}, \label{2.2.i}\\
\mu\frac{\partial H_{yx}}{\partial
t}+\sigma_{x}^{\ast}H_{yx}&=&\frac{\partial (E_{zx}+E_{zy})}{\partial
x}, \label{2.2.j}\\
\mu\frac{\partial H_{zx}}{\partial
t}+\sigma_{x}^{\ast}H_{zx}&=&-\frac{\partial (E_{yz}+E_{yx})}{\partial
x}, \label{2.2.k}\\
\mu\frac{\partial H_{zy}}{\partial
t}+\sigma_{y}^{\ast}H_{zy}&=&\frac{\partial (E_{xy}+E_{xz})}{\partial
y}, \label{2.2.l}
\end{eqnarray}
where the parameters $(\sigma_{x}, \sigma_{y},
\sigma_{z}, \sigma_{x}^{\ast}, \sigma_{y}^{\ast},
\sigma_{z}^{\ast})$ are homogeneous to electric and magnetic
conductivities.

If $\sigma_{x}=\sigma_{y}=\sigma_{z}$ and
$\sigma_{x}^{\ast}=\sigma_{y}^{\ast}=\sigma_{z}^{\ast}=0$, then
(\ref{2.2.a})-(\ref{2.2.l}) yield the classical Maxwell equations
(\ref{2.1.1})-(\ref{2.1.2}). Thus, the absorbing medium defined by (\ref{2.2.a})-(\ref{2.2.l})
 holds as particular cases of all usual media
(vacuum, conductive media).

The 3D PML technique is a straightforward generalization of the 2D
case \cite{RefP}. The Maxwell equations are solved by the FDTD
method within a computational domain surrounded by an absorbing
layer which is an aggregate of PML media.

\section{One formulation equivalent to Berenger's formulation}
\label{sec:4}Berenger's formulation involves a splitting of the
unknown electromagnetic fields. The idea of \cite{RefR} is to
restore the usual operator by introducing a new variable.

Let us consider the Berenger's PML parrel $xoy$-plane, i.e.,
$\sigma_{x}=\sigma_{x}^{\ast}=0$ and
$\sigma_{y}=\sigma_{y}^{\ast}=0$. For simplicity, we assume that
$\varepsilon=\mu \equiv1$ and
$\sigma_{z}=\sigma_{z}^{\ast}\equiv\sigma$, where $\sigma$ is a
constant that does not depend on $x$, $y$, $z$, $t$. Then Berenger's
systems (\ref{2.2.a})-(\ref{2.2.l}) of the 3D Maxwell equations can be
rewritten as
\begin{eqnarray}
\frac{\partial E_{xy}}{\partial
t}&=&\frac{\partial (H_{zx}+H_{zy})}{\partial y}, \label{3.a}\\
\frac{\partial E_{xz}}{\partial
t}+\sigma E_{xz}&=&-\frac{\partial (H_{yz}+H_{yx})}{\partial z},\label{3.b}\\
\frac{\partial E_{yz}}{\partial
t}+\sigma E_{yz}&=&\frac{\partial (H_{xy}+H_{xz})}{\partial z},\label{3.c}\\
\frac{\partial E_{yx}}{\partial t}&=&-\frac{\partial
(H_{zx}+H_{zy})}{\partial x},\label{3.d}\\
\frac{\partial E_{zx}}{\partial
t}&=&\frac{\partial (H_{yz}+H_{yx})}{\partial x}, \label{3.e}\\
\frac{\partial E_{zy}}{\partial
t}&=&-\frac{\partial (H_{xy}+H_{xz})}{\partial y}, \label{3.f}\\
\frac{\partial H_{xy}}{\partial
t}&=&-\frac{\partial (E_{zx}+E_{zy})}{\partial y}, \label{3.g}\\
\frac{\partial H_{xz}}{\partial
t}+\sigma H_{xz}&=&\frac{\partial (E_{yz}+E_{yx})}{\partial z},\label{3.h}\\
\frac{\partial H_{yz}}{\partial
t}+\sigma H_{yz}&=&-\frac{\partial (E_{xy}+E_{xz})}{\partial z}, \label{3.i}\\
\frac{\partial H_{yx}}{\partial
t}&=&\frac{\partial (E_{zx}+E_{zy})}{\partial x}, \label{3.j}\\
\frac{\partial H_{zx}}{\partial
t}&=&-\frac{\partial (E_{yz}+E_{yx})}{\partial x}, \label{3.k}\\
\frac{\partial H_{zy}}{\partial
t}&=&\frac{\partial (E_{xy}+E_{xz})}{\partial y}. \label{3.l}
\end{eqnarray}

Let $E_{xy}, E_{xz}, E_{yx}, E_{yz}, E_{zx}, E_{zy}, H_{xy}, H_{xz},
H_{yx}, H_{yz}, H_{zx}, H_{zy}$ be a solution of Berenger's system
(\ref{3.a})-(\ref{3.l}) with initial conditions $\textbf{E}^{0}$,
$\textbf{H}^{0}$. Adding (\ref{3.e}) and (\ref{3.f}), (\ref{3.k}) and (\ref{3.l}),
respectively, we can obtain that
\begin{eqnarray}
\frac{\partial E_{z}}{\partial
t}&=&\frac{\partial H_{y}}{\partial x}-\frac{\partial H_{x}}{\partial
y}, \label{3.1}\\
\frac{\partial H_{z}}{\partial
t}&=&\frac{\partial E_{x}}{\partial y}-\frac{\partial E_{y}}{\partial
x}. \label{3.2}
\end{eqnarray}

Applying $\partial_{t}+\sigma$ to (\ref{3.a}), $\partial_{t}$ to
(\ref{3.b}), and adding the two terms give
\begin{equation}
\partial_{t}(\partial_{t}+\sigma)E_{xy}+\partial_{t}(\partial_{t}+\sigma)E_{xz}+\partial_{t}\partial_{z}H_{y}-(\partial_{t}+\sigma)\partial_{y}H_{z}=0.\nonumber
\end{equation}

Since $\sigma$ does not depend on $y$, the operators
$\partial_{t}+\sigma$ and $\partial_{y}$ commute, we get by setting
$E_{x}=E_{xy}+E_{xz}$
\begin{equation}
\partial_{t}[(\partial_{t}+\sigma)E_{x}+\partial_{z}H_{y}]-\partial_{y}(\partial_{t}+\sigma)H_{z}=0.\nonumber
\end{equation}

In order to transform the last term into a time derivative, we
introduce a new variable $\widetilde{H_{z}}$ satisfying
\begin{equation}\label{3.3}
(\partial_{t}+\sigma)H_{z}=\partial_{t}\widetilde{H_{z}},
\end{equation}
then
\begin{equation}
\partial_{t}[(\partial_{t}+\sigma)E_{x}+\partial_{z}H_{y}-\partial_{y}\widetilde{H_{z}}]=0.\nonumber
\end{equation}

If we make the assumption that at $t=0$,
\begin{equation}\label{3.4}
(\partial_{t}E_{x})^{0}+\sigma E_{x}^{0}+\partial_{z}H_{y}^{0}-\partial_{y}\widetilde{H_{z}^{0}}=0,
\end{equation}
it follows
\begin{equation}\label{3.5}
(\partial_{t}+\sigma)E_{x}+\partial_{z}H_{y}-\partial_{y}\widetilde{H_{z}}=0.
\end{equation}

Note that (\ref{3.3}) can not completely determine $\widetilde{H_{z}}$. We
have to prescribe an initial value of $\widetilde{H_{z}}$, say
\begin{equation}\label{3.6}
\widetilde{H_{z}^{0}}=H_{z}^{0},
\end{equation}
which in particular implies that $\widetilde{H_{z}}\equiv H_{z}$ if
$\sigma=0$.

A similar process, it implies
\begin{equation}\label{3.7}
(\partial_{t}+\sigma)E_{y}+\partial_{x}H_{z}-\partial_{z}\widetilde{H_{x}}=0.
\end{equation}

Similarly, it follows from introducing a new variable
$\widetilde{E_{z}}$ satisfying
\begin{equation}\label{3.8}
(\partial_{t}+\sigma)E_{z}=\partial_{t}\widetilde{E_{z}},
\end{equation}
that
\begin{eqnarray}
(\partial_{t}+\sigma)H_{x}-\partial_{z}E_{y}+\partial_{y}\widetilde{E_{z}}&=&0, \label{3.9}\\
(\partial_{t}+\sigma)H_{y}+\partial_{z}E_{x}-\partial_{x}\widetilde{E_{z}}&=&0.\label{3.10}
\end{eqnarray}

In order to make the calculations of
 the next sections more readable, it is useful to
adopt new notations for $E_{z}$ and $\widetilde{E_{z}}$, $H_{z}$ and
$\widetilde{H_{z}}$.\\

 $\bullet$ $E_{z}$ is denoted by $E_{z}^{\ast}$, $\widetilde{E_{z}}$ is denoted by $E_{z}$;\\

 $\bullet$ $H_{z}$ is denoted by $H_{z}^{\ast}$, $\widetilde{H_{z}}$ is denoted by $H_{z}$.\\

 Then the un-splitting formulation (\ref{3.a})-(\ref{3.l}) can be rewritten as
\begin{equation}\label{3.11}
\left\{
\begin{array}{ccccccccc}
(\partial_{t}+\sigma)E_{x}+\partial_{z}H_{y}-\partial_{y}H_{z}=0,~~~~(a)\\
(\partial_{t}+\sigma)E_{y}+\partial_{x}H_{z}-\partial_{z}H_{x}=0,~~~~(b)\\
\partial_{t}E_{z}^{\ast}+\partial_{y}H_{x}-\partial_{x}H_{y}=0,~~~~~~~~~~~~~(c)\\
(\partial_{t}+\sigma)E_{z}^{\ast}=\partial_{t}E_{z},~~~~~~~~~~~~~~~~~~~~~~(d)\\
(\partial_{t}+\sigma)H_{x}+\partial_{y}E_{z}-\partial_{z}E_{y}=0,~~~~(e)\\
(\partial_{t}+\sigma)H_{y}+\partial_{z}E_{x}-\partial_{x}E_{z}=0,~~~~(f)\\
\partial_{t}H_{z}^{\ast}+\partial_{x}E_{y}-\partial_{y}E_{x}=0,~~~~~~~~~~~~~(g)\\
(\partial_{t}+\sigma)H_{z}^{\ast}=\partial_{t}H_{z}.~~~~~~~~~~~~~~~~~~~~~(h)\\
\end{array}\right.
\end{equation}

\section{Energy evolution of the multi-symplectic methods for 3D Maxwell equations}
\label{sec:5}In this section, we apply two multi-symplectic methods
to discrete
3D Maxwell equations.\\

1) Multi-symplectic Yee method.

We introduce the difference operators
 ($k=n$, or $k=n+\frac{1}{2}$)
\begin{eqnarray}
(D_{\Delta t}U)^{k}&=&\frac{U^{k+\frac{1}{2}}-U^{k-\frac{1}{2}}}{\Delta t},~~(D_{\Delta x}U)_{k}=\frac{U_{k+\frac{1}{2}}-U_{k-\frac{1}{2}}}{\Delta x},\nonumber\\
(D_{\Delta y}U)_{k}&=&\frac{U_{k+\frac{1}{2}}-U_{k-\frac{1}{2}}}{\Delta y},~~(D_{\Delta z}U)_{k}=\frac{U_{k+\frac{1}{2}}-U_{k-\frac{1}{2}}}{\Delta z},\nonumber\\
\overline{U}^{k}&=&\frac{U^{k+\frac{1}{2}}+U^{k-\frac{1}{2}}}{2}, \nonumber\\
(D_{\Delta t}^{\sigma}U)_{\alpha, \beta, \gamma}^{n}&=&\frac{U^{n+\frac{1}{2}}_{\alpha,\beta,\gamma}-U^{n-\frac{1}{2}}_{\alpha,\beta,\gamma}}{\Delta t}
+\sigma\frac{U_{\alpha, \beta, \gamma}^{n+\frac{1}{2}}+U_{\alpha, \beta, \gamma}^{n-\frac{1}{2}}}{2}\nonumber\\
&=&
(D_{\Delta t}U)_{\alpha, \beta,
\gamma}^{n}+\sigma\overline{U}_{\alpha, \beta, \gamma}^{n}. \label{4.3}
\end{eqnarray}

With these notations, multi-symplectic Yee method for (\ref{3.11}) is
\begin{small}
\begin{eqnarray}
(D_{\Delta t}^{\sigma}E_{x})_{i+1/2,j,k}^{n+1/2}+(D_{\Delta
z}H_{y})_{i+1/2,j,k}^{n+1/2}&=&(D_{\Delta
y}H_{z})_{i+1/2,j,k}^{n+1/2},\label{dis1}\\
(D_{\Delta t}^{\sigma}E_{y})_{i,j+1/2,k}^{n+1/2}+(D_{\Delta
x}H_{z})_{i,j+1/2,k}^{n+1/2}&=&(D_{\Delta
z}H_{x})_{i,j+1/2,k}^{n+1/2},\label{dis2}\\
(D_{\Delta t}E_{z}^{{\ast}})_{i,j,k+1/2}^{n+1/2}+(D_{\Delta y}H_{x})_{i,j,k+1/2}^{n+1/2}&=&(D_{\Delta x}H_{y})_{i,j,k+1/2}^{n+1/2},\label{dis3}\\
(D_{\Delta t}^{\sigma}E_{z}^{{\ast}})_{i,j,k+1/2}^{n+1/2}&=&(D_{\Delta
t}E_{z})_{i,j,k+1/2}^{n+1/2},\label{dis4}\\
(D_{\Delta t}^{\sigma}H_{x})_{i,j+1/2,k+1/2}^{n}+(D_{\Delta
y}E_{z})_{i,j+1/2,k+1/2}^{n}&=&(D_{\Delta
z}E_{y})_{i,j+1/2,k+1/2}^{n},\label{dis5}\\
(D_{\Delta t}^{\sigma}H_{y})_{i+1/2,j,k+1/2}^{n}+(D_{\Delta
z}E_{x})_{i+1/2,j,k+1/2}^{n}&=&(D_{\Delta
x}E_{z})_{i+1/2,j,k+1/2}^{n},\label{dis6}\\
(D_{\Delta t}H_{z}^{{\ast}})_{i+1/2,j+1/2,k}^{n}+(D_{\Delta x}E_{y})_{i+1/2,j+1/2,k}^{n}&=&(D_{\Delta y}E_{x})_{i+1/2,j+1/2,k}^{n},\label{dis7}\\
(D_{\Delta t}^{\sigma}H_{z}^{{\ast}})_{i+1/2,j+1/2,k}^{n}&=&(D_{\Delta
t}H_{z})_{i+1/2,j+1/2,k}^{n}.\label{dis8}
\end{eqnarray}
\end{small}

We define the one-dimensional (1D) discrete scalar product
$$(U,V)_{h}=\sum_{\alpha}U_{\alpha}V_{\alpha},~~~~~~\forall(U, V)\in (l^{2}(\alpha))^{2},$$\\
where $\alpha$ is either an integer or a half-integer. The 3D
discrete scalar product will be denoted $(((U, V)))_{h}$ (or when
needed with an index $l^{2}(\alpha)\times l^{2}(\beta)\times
l^{2}(\gamma)$).

Discrete integrations by parts yields
\begin{small}
\begin{eqnarray}\label{4.13}
\left\{
\begin{array}{ccccc}
(D_{\Delta x}U,V)_{l^{2}(\alpha)}&=&-(U,D_{\Delta
x}V)_{l^{2}(\alpha+1/2)}, \forall U\in l^{2}(\alpha+1/2), V\in l^{2}(\alpha),\label{lemma01}\\
(D_{\Delta y}U,V)_{l^{2}(\beta)}&=&-(U,D_{\Delta
y}V)_{l^{2}(\beta+1/2)},\forall U\in l^{2}(\beta+1/2), V\in
l^{2}(\beta),\label{02}\\
(D_{\Delta z}U,V)_{l^{2}(\gamma)}&=&-(U,D_{\Delta
z}V)_{l^{2}(\gamma+1/2)},\forall U\in l^{2}(\gamma+1/2), V\in
l^{2}(\gamma),\label{03}
\end{array}\right.
\end{eqnarray}
\end{small}
\begin{small}
\begin{eqnarray}\label{4.14}
((((D_{\Delta
t}U)^{n+1/2},\frac{U^{n}+U^{n+1}}{2})))_{h}\nonumber\\
=\frac{1}{2\Delta
t}(\|U^{n+1}\|_{h}^{2}-\|U^{n}\|_{h}^{2}),~~\forall U\in
l^{2}(\alpha)\times l^{2}(\beta)\times l^{2}(\gamma).
\end{eqnarray}
\end{small}

Since $\sigma$ is constant,
\begin{eqnarray}
D_{\Delta t}^{\sigma}D_{\Delta t}&=&D_{\Delta t}D_{\Delta t}^{\sigma},\label{4.15}\\
D_{\Delta t}^{\sigma}D_{\Delta x}&=&D_{\Delta x}D_{\Delta t}^{\sigma},\label{4.16}\\
D_{\Delta t}^{\sigma}D_{\Delta y}&=&D_{\Delta y}D_{\Delta t}^{\sigma},\label{4.17}\\
D_{\Delta t}^{\sigma}D_{\Delta z}&=&D_{\Delta z}D_{\Delta t}^{\sigma}.\label{4.18}
\end{eqnarray}

\begin{theorem}\label{them4.1}

For any integer $n\geq0$, let
$\textbf{E}_{i,j,k}^{n}=(E_{x_{i,j,k}}^{n}, E_{y_{i,j,k}}^{n},
E_{z_{i,j,k}}^{n} )$ and $\textbf{H}_{i,j,k}^{n}=(H_{x_{i,j,k}}^{n},
H_{y_{i,j,k}}^{n}, H_{z_{i,j,k}}^{n})$ be the solution of
(\ref{dis1})-(\ref{dis8}), then the discrete version of the energy evolution
law is,
\begin{eqnarray}\label{4.19}
\varepsilon_{1}^{n+1/2}&-&\varepsilon_{1}^{n-1/2}+2\sigma \|\overline{D_{\Delta t}E_{x}^{n}}\|_{h}^{2}+2\sigma \|\overline{D_{\Delta t}E_{y}^{n}}\|_{h}^{2}\nonumber\\
&+&\sigma (((\overline{\overline{D_{\Delta x}E_{y}^{n}}}-\overline{\overline{D_{\Delta y}E_{x}^{n}}},D_{\Delta t}^{\sigma}H_{z}^{n})))_{h}=0,
\end{eqnarray}
where
\begin{small}
$$\varepsilon_{1}^{n+1/2}=\frac{1}{2\Delta t}\Big(\|(D_{\Delta t}E_{x})^{n+1/2}\|_{h}^{2}+\|(D_{\Delta t}E_{y})^{n+1/2}+\|(D_{\Delta t}E_{z})^{n+1/2}\|_{h}^{2}~~~~~~~~~~~~~~$$
$$+\|\sigma \frac{E_{x}^{n}+E_{x}^{n+1}}{2}\|_{h}^{2}+\|\sigma \frac{E_{y}^{n}+E_{y}^{n+1}}{2}\|_{h}^{2}$$$$+((((D_{\Delta t}^{\sigma}H_{x})^{n+1},(D_{\Delta t}^{\sigma}H_{x})^{n})))_{h})+((((D_{\Delta t}^{\sigma}H_{y})^{n+1},(D_{\Delta t}^{\sigma}H_{y})^{n})))_{h}\Big).$$\\
\end{small}
\end{theorem}
\begin{proof}
We divide it into seven parts\\

\textbf{(i)} By applying $D_{\Delta t}^{\sigma}$ to (\ref{dis1}), we
get
\begin{equation}\label{4.20}
((D_{\Delta t}^{\sigma})^{2}E_{x})^{n}+(D_{\Delta z}D_{\Delta t}^{\sigma}H_{y})^{n}-(D_{\Delta y}D_{\Delta t}^{\sigma}H_{z})^{n}=0.
\end{equation}
We multiply (\ref{4.20}) with $\overline{(D_{\Delta t}E_{x})}^{n}$ to get
\begin{small}
\begin{equation}\label{4.21}
(((((D_{\Delta t}^{\sigma})^{2}E_{x})^{n}+(D_{\Delta z}D_{\Delta t}^{\sigma}H_{y})^{n}-(D_{\Delta y}D_{\Delta t}^{\sigma}H_{z})^{n},\overline{(D_{\Delta t}E_{x})}^{n})))_{h}=0.
\end{equation}
\end{small}

\textbf{(ii)} By applying $D_{\Delta t}^{\sigma}$ to (\ref{dis2}).
Since $\sigma$ is constant, we get
\begin{equation}\label{4.22}
((D_{\Delta t}^{\sigma})^{2}E_{y})^{n}+(D_{\Delta x}D_{\Delta t}^{\sigma}H_{z})^{n}-(D_{\Delta z}D_{\Delta t}^{\sigma}H_{x})^{n}=0.
\end{equation}
Multiplying (\ref{4.22}) with $\overline{(D_{\Delta t}E_{y})}^{n}$ yields
\begin{small}
\begin{equation}\label{4.23}
(((((D_{\Delta t}^{\sigma})^{2}E_{y})^{n}+(D_{\Delta x}D_{\Delta t}^{\sigma}H_{z})^{n}-(D_{\Delta z}D_{\Delta t}^{\sigma}H_{x})^{n},\overline{(D_{\Delta t}E_{y})}^{n})))_{h}=0.
\end{equation}
\end{small}

\textbf{(iii)} After applying $D_{\Delta t}^{\sigma}$ to (\ref{dis3}),
then the equation (\ref{dis3}) be shift to time $n$
\begin{equation}
D_{\Delta t}^{\sigma}D_{\Delta t}E_{z}^{\ast}+D_{\Delta t}^{\sigma}D_{\Delta y}H_{x}-D_{\Delta t}^{\sigma}D_{\Delta x}H_{y}=0.\nonumber
\end{equation}
 Using (\ref{4.17}),(\ref{4.18}) and (\ref{dis4}), this is equivalent to
 \begin{equation}\label{4.24}
(D_{\Delta t}^{2}E_{z})^{n}+(D_{\Delta y}D_{\Delta t}^{\sigma}H_{x})^{n}-(D_{\Delta x}D_{\Delta t}^{\sigma}H_{y})^{n}=0. \end{equation}
We multiply (\ref{4.24}) with $\overline{(D_{\Delta t}E_{z})}^{n}$ to get
\begin{small}
\begin{equation}\label{4.25}
((((D_{\Delta t}^{2}E_{z})^{n}+(D_{\Delta y}D_{\Delta t}^{\sigma}H_{x})^{n}-(D_{\Delta x}D_{\Delta t}^{\sigma}H_{y})^{n},\overline{(D_{\Delta t}E_{z})}^{n})))_{h}=0.
\end{equation}
\end{small}

\textbf{(iv)} Equation (\ref{dis5}) is written as (\ref{dis1}),(\ref{dis2}) and
(\ref{dis3}) not at time $n+1/2$ but at time $n$. It is thus necessary to
consider the mean-value of (\ref{dis5}) at $n$ and $n+1$:
\begin{equation}
\overline{(D_{\Delta t}^{\sigma}H_{x})}^{n+1/2}+\overline{(D_{\Delta y}E_{z})}^{n+1/2}-\overline{(D_{\Delta z}E_{y})}^{n+1/2}=0.\nonumber
\end{equation}
We apply $D_{\Delta t}$ to this equation and multiply it by
$(D_{\Delta t}^{\sigma}H_{x})^{n}$ to get
\begin{small}
\begin{eqnarray}\label{4.26}
(((\overline{D_{\Delta t}(D_{\Delta t}^{\sigma}H_{x})}^{n+1/2}&+&\overline{(D_{\Delta y}D_{\Delta t}E_{z})}^{n+1/2}\nonumber\\
&-&\overline{(D_{\Delta z}D_{\Delta t}E_{y})}^{n+1/2},(D_{\Delta t}^{\sigma}H_{x})^{n})))_{h}=0.
\end{eqnarray}
\end{small}

\textbf{(v)} In the same way, we consider the mean-value of (\ref{dis6}) at
$n$ and $n+1$:
\begin{equation}
\overline{(D_{\Delta t}^{\sigma}H_{y})}^{n+1/2}+\overline{(D_{\Delta z}E_{x})}^{n+1/2}-\overline{(D_{\Delta x}E_{z})}^{n+1/2}=0.\nonumber
\end{equation}
We apply $D_{\Delta t}$ to this equation and multiply it by
$(D_{\Delta t}^{\sigma}H_{y})^{n}$ to get
\begin{small}
\begin{eqnarray}\label{4.27}
(((\overline{D_{\Delta t}(D_{\Delta t}^{\sigma}H_{y})}^{n+1/2}&+&\overline{(D_{\Delta z}D_{\Delta t}E_{x})}^{n+1/2}\nonumber\\
&-&\overline{(D_{\Delta x}D_{\Delta t}E_{z})}^{n+1/2},(D_{\Delta t}^{\sigma}H_{y})^{n})))_{h}=0.
\end{eqnarray}
\end{small}

\textbf{(vi)} Let us consider the mean-value of (\ref{dis7}) at $n$
and $n+1$:
\begin{equation}
\overline{(D_{\Delta t}^{\sigma}H_{z}^{\ast})}^{n+1/2}+\overline{(D_{\Delta x}E_{y})}^{n+1/2}-\overline{(D_{\Delta
y}E_{x})}^{n+1/2}=0.\nonumber
\end{equation}

 Using (\ref{4.16}), (\ref{4.17}) and (\ref{dis8}), this is
equivalent to
\begin{equation}\label{4.28}
\overline{(D_{\Delta t}^{2}H_{z})}^{n+1/2}+\overline{(D_{\Delta x}E_{y})}^{n+1/2}-\overline{(D_{\Delta y}E_{x})}^{n+1/2}=0.
\end{equation}

Then, we multiply (\ref{4.28}) with $D_{\Delta t}^{\sigma}H_{z}^{n}$ to
get
\begin{small}
\begin{equation}\label{4.29}
(((\overline{(D_{\Delta t}^{2}H_{z})}^{n+1/2}+\overline{(D_{\Delta x}E_{y})}^{n+1/2}-\overline{(D_{\Delta y}E_{x})}^{n+1/2},D_{\Delta t}^{\sigma}H_{z}^{n})))_{h}.
\end{equation}
\end{small}

\textbf{(vii)} Adding (\ref{4.21}), (\ref{4.23}), (\ref{4.25}), (\ref{4.26}), (\ref{4.27}) and
(\ref{4.29}). Due to (\ref{4.13}) we obtain that
\begin{small}
\begin{eqnarray}\label{4.30}
&&(((((D_{\Delta t}^{\sigma})^{2}E_{x})^{n},\overline{(D_{\Delta t}E_{x})}^{n})))_{h}+(((((D_{\Delta t}^{\sigma})^{2}E_{y})^{n},\overline{(D_{\Delta t}E_{y})}^{n})))_{h}\nonumber\\[2mm]
&+&((((\overline{(D_{\Delta t}D_{\Delta t}^{\sigma}H_{x})}^{n},(D_{\Delta t}^{\sigma}H_{x})^{n}))))_{h}+((((\overline{(D_{\Delta t}D_{\Delta t}^{\sigma}H_{y})}^{n},(D_{\Delta t}^{\sigma}H_{y})^{n}))))_{h}\nonumber\\[2mm]
&+&(((D_{\Delta t}^{2}E_{z})^{n},\overline{(D_{\Delta t}E_{z})}^{n}))_{h}+\sigma (((\overline{\overline{D_{\Delta x}E_{y}^{n}}}-\overline{\overline{D_{\Delta y}E_{x}^{n}}},D_{\Delta t}^{\sigma}H_{z}^{n})))_{h}=0.
\end{eqnarray}
\end{small}

 Thus, we put it into six parts, just as follows:
  \begin{eqnarray}\label{4.31}
\left\{
\begin{array}{ccccc}
S^{1}&+&S^{2}+S^{3}+S^{4}+S^{5}+S^{6}=0,\\[2mm]
S^{1}&=&(((D_{\Delta t}^{2}E_{z})^{n},\overline{(D_{\Delta t}E_{z})}^{n}))_{h},\\[2mm]
S^{2}&=&((((D_{\Delta t}^{\sigma})^{2}E_{x})^{n},\overline{(D_{\Delta t}E_{x})}^{n}))_{h},\\[2mm]
S^{3}&=&((((D_{\Delta t}^{\sigma})^{2}E_{y})^{n},\overline{(D_{\Delta t}E_{y})}^{n}))_{h},\\[2mm]
S^{4}&=&((\overline{(D_{\Delta t}D_{\Delta
t}^{\sigma}H_{x})}^{n},(D_{\Delta t}^{\sigma}H_{x})^{n}))_{h},\\[2mm]
S^{5}&=&((\overline{(D_{\Delta t}D_{\Delta
t}^{\sigma}H_{y})}^{n},(D_{\Delta t}^{\sigma}H_{y})^{n}))_{h},\\[2mm]
S^{6}&=&\sigma (((\overline{\overline{D_{\Delta
x}E_{y}^{n}}}-\overline{\overline{D_{\Delta y}E_{x}^{n}}},D_{\Delta
t}^{\sigma}H_{z}^{n})))_{h}.
\end{array}\right.
\end{eqnarray}

We now calculate $S^{j}(j=1, 2, 3,4,5)$ explicitly. From (\ref{4.14}), it
is straightforward to get
\begin{equation}\label{4.32}
S^{1}=\frac{1}{2\Delta t}\Big(\|(D_{\Delta t}E_{z})^{n+1/2}\|_{h}^{2}-\|(D_{\Delta t}E_{z})^{n-1/2}\|_{h}^{2}\Big). \end{equation}

There is also no difficulty to rewrite the fourth and fifth terms as
\begin{small}
\begin{eqnarray}
S^{4}=\frac{1}{2\Delta t}(((((D_{\Delta t}^{\sigma}H_{x})^{n+1},(D_{\Delta t}^{\sigma}H_{x})^{n})))_{h}-((((D_{\Delta t}^{\sigma}H_{x})^{n},(D_{\Delta t}^{\sigma}H_{x})^{n-1})))_{h}),\nonumber\\[2mm]
S^{5}=\frac{1}{2\Delta t}(((((D_{\Delta t}^{\sigma}H_{y})^{n+1},(D_{\Delta t}^{\sigma}H_{y})^{n})))_{h}-((((D_{\Delta t}^{\sigma}H_{y})^{n},(D_{\Delta t}^{\sigma}H_{y})^{n-1})))_{h}).\nonumber
\end{eqnarray}
\end{small}

Concerning the second term, we first develop
\begin{eqnarray}\label{p23}
((D_{\Delta t}^{\sigma})^{2}E_{x})^{n}&=& (D_{\Delta t}^{\sigma}D_{\Delta t}E_{x})^{n}+\sigma \overline{(D_{\Delta
t}^{\sigma}E_{x})}^{n}\nonumber\\
&=&(D_{\Delta t}^{2}E_{x})^{n}+2\sigma \overline{(D_{\Delta
t}E_{x})}^{n}+\sigma^{2}\frac{\overline{E_{x}}^{n+1/2}+\overline{E_{x}}^{n-1/2}}{2}.\nonumber
\end{eqnarray}

We multiply this expression with $\overline{(D_{\Delta
t}E_{x})}^{n}$, and rearrange the last term
\begin{eqnarray}
&\sigma^{2}&((\frac{\overline{E_{x}}^{n+1/2}+\overline{E_{x}}^{n-1/2}}{2},\overline{(D_{\Delta t}E_{x})}^{n}))_{h}\nonumber\\
&=&\frac{\sigma^{2}}{2\Delta t}(\|\overline{E_{x}}^{n+1/2}\|_{h}^{2}-\|\overline{E_{x}}^{n-1/2}\|_{h}^{2}).\nonumber
\end{eqnarray}
and get
\begin{small}
\begin{eqnarray}\label{4.35}
S^{2}&=&\frac{1}{2\Delta t}(\|(D_{\Delta t}E_{x})^{n+1/2}\|_{h}^{2}-\|(D_{\Delta
t}E_{x})^{n-1/2}\|_{h}^{2}\nonumber\\
&+&\sigma^{2}\|\overline{E_{x}}^{n+1/2}\|_{h}^{2}-\sigma^{2}\|\overline{E_{x}}^{n-1/2}\|_{h}^{2})+2\sigma
\|\overline{(D_{\Delta t}E_{x})}^{n}\|_{h}^{2}.
\end{eqnarray}
\end{small}

The same as $S^{3}$, we can get
\begin{small}
\begin{eqnarray}\label{4.36}
S^{3}&=&\frac{1}{2\Delta t}(\|(D_{\Delta t}E_{y})^{n+1/2}\|_{h}^{2}-\|(D_{\Delta
t}E_{y})^{n-1/2}\|_{h}^{2}\nonumber\\
&+&\sigma^{2}\|\overline{E_{y}}^{n+1/2}\|_{h}^{2}-\sigma^{2}\|\overline{E_{y}}^{n-1/2}\|_{h}^{2})+2\sigma
\|\overline{(D_{\Delta t}E_{y})}^{n}\|_{h}^{2}.
\end{eqnarray}
\end{small}

 Due to $S^{1}+S^{2}+S^{3}+S^{4}+S^{5}+S^{6}=0$, then
we can obtain the energy evolution law (\ref{4.19}).
 The proof
is finished.
\end{proof}

2) General multi-symplectic Runge-Kutta methods.

The Runge-Kutta methods for Maxwell equations (\ref{2.1.7}) are
\begin{eqnarray}
Z_{i,j,k,n}&=&z^{0}_{i,j,k}+\Delta t\sum^{r}_{s=1}a_{n,s}\partial_{t}Z_{i,j,k,s}, \label{4.37.a}\\
z^{1}_{i,j,k}&=&z^{0}_{i,j,k}+\Delta t\sum^{r}_{s=1}b_{s}\partial_{t}Z_{i,j,k,s}, \label{4.37.b}\\
Z_{i,j,k,n}&=&z^{n}_{0,j,k}+\Delta x\sum^{r_{1}}_{u=1}\overline{a}_{i,u}\partial_{x}Z_{u,j,k,n}, \label{4.37.c}\\
z^{n}_{1,j,k}&=&z^{n}_{0,j,k}+\Delta x\sum^{r_{1}}_{u=1}\overline{b}_{u}\partial_{x}Z_{u,j,k,n}, \label{4.37.d}\\
Z_{i,j,k,n}&=&z^{n}_{i,0,k}+\Delta y\sum^{r_{2}}_{v=1}\widetilde{a}_{j,v}\partial_{y}Z_{i,v,k,n}, \label{4.37.e}\\
z^{n}_{i,1,k}&=&z^{n}_{i,0,k}+\Delta y\sum^{r_{2}}_{v=1}\widetilde{b}_{v}\partial_{y}Z_{i,v,k,n}, \label{4.37.f}\\
Z_{i,j,k,n}&=&z^{n}_{i,j,0}+\Delta z\sum^{r_{3}}_{w=1}\widehat{a}_{k,w}\partial_{z}Z_{i,j,w,n}, \label{4.37.g}\\
z^{n}_{i,j,1}&=&z^{n}_{i,j,0}+\Delta z\sum^{r_{3}}_{w=1}\widehat{b}_{w}\partial_{z}Z_{i,j,w,n}. \label{4.37.h}
\end{eqnarray}
\begin{theorem}[Hong et al. 2005]\label{them4.2}
If in the methods (\ref{4.37.a})-(\ref{4.37.h})
\begin{eqnarray}
b_{s}b_{n}-b_{s}a_{s,n}-b_{n}a_{n,s}&=&0, \label{4.38}\\
\overline{b}_{u}\overline{b}_{i}-\overline{b}_{u}\overline{a}_{u,i}-\overline{b}_{i}\overline{a}_{i,u}&=&0, \label{4.39}\\
\widetilde{b}_{v}\widetilde{b}_{j}-\widetilde{b}_{v}\widetilde{a}_{v,j}-\widetilde{b}_{j}\widetilde{a}_{j,v}&=&0,\label{4.40}
\end{eqnarray}
and
\begin{equation}\label{4.41}
\widehat{b}_{w}\widehat{b}_{k}-\widehat{b}_{w}\widehat{a}_{w,k}-\widehat{b}_{k}\widehat{a}_{k,w}=0,
\end{equation}
hold for $s,n=1,2,...,r$, $u,i=1,2,...,r_{1}$, $v,j=1,2,...,r_{2}$
and $w,k=1,2,...,r_{3}$, then the method (\ref{4.37.a})-(\ref{4.37.h}) is
multi-symplectic.
\end{theorem}

Let $U=(E_{xy},E_{xz},E_{yz},E_{yx},E_{xy},
E_{xz},H_{xy},H_{xz},H_{yz},H_{yx},H_{xy},H_{xz})^{T}$, then the 3D
Maxwell equations with PML (\ref{3.a})-(\ref{3.l}) can be rewritten as
\begin{equation}\label{4.42}
\partial_{t}U+\Sigma U=P(\nabla)U,
\end{equation}
with
$$\Sigma=\left(\begin{array}{ccccccccccccc}
A_{6\times6}&0\\
0&A_{6\times6}\\
\end{array}\right),~~~~P(\nabla)=\left(\begin{array}{ccccccccccccc}
0&B_{6\times6}\\
-B_{6\times6}&0\\
\end{array}\right),$$
where
$$A_{6\times6}=\left(\begin{array}{ccccccc}
0&&&&&&\\
&\sigma&&&&&\\
&&\sigma&&&&\\
&&&0&&&\\
&&&&0&&\\
&&&&&0&\\
\end{array}\right),~~~~~~~~~~~$$
$$B_{6\times6}=\left(\begin{array}{ccccccccc}
&&&&\frac{\partial}{\partial y}&\frac{\partial}{\partial y}&\\
&&-\frac{\partial}{\partial z}&-\frac{\partial}{\partial z}&&&\\
\frac{\partial}{\partial z}&\frac{\partial}{\partial z}&&&&&\\
&&&&-\frac{\partial}{\partial x}&-\frac{\partial}{\partial x}&\\
&&\frac{\partial}{\partial x}&\frac{\partial}{\partial x}&&&\\
-\frac{\partial}{\partial y}&-\frac{\partial}{\partial y}&&&&&\\
\end{array}\right). $$

Based on theorem (\ref{them4.2}), we apply an $s$-stage symplectic Runge-Kutta
method to the $t$-direction,  of form (\ref{4.42}), to obtain the
following scheme,
\begin{eqnarray}
U^{n}&=&u^{0}+\Delta t\sum_{m=1}^{s}a_{nm}(-\Sigma U^{m}+P(\nabla)U^{m}),\label{4.44}\\
u^{1}&=&u^{0}+\Delta t\sum_{m=1}^{s}b_{m}(-\Sigma U^{m}+P(\nabla)U^{m}),\label{4.45}
\end{eqnarray}
where the coefficients of the equations (\ref{4.44})-(\ref{4.45}) satisfy the
following symplectic condition:
\begin{equation}\label{4.46}
b_{m}a_{mn}+b_{n}a_{nm}=b_{m}b_{n}, ~~~~ m,n=1,2,...,s.
\end{equation}

Now, we introduce the following difference operators and apply Yee
method to the $x$-direction, $y$-direction and $z$-direction
respectively, of form (\ref{4.44})-(\ref{4.45}).
\begin{eqnarray}
\delta_{x}V_{i,j,k}^{n}&=&\frac{V_{i+1/2,j,k}^{n}-V_{i-1/2,j,k}^{n}}{\Delta x},\nonumber\\
\delta_{y}V_{i,j,k}^{n}&=&\frac{V_{i,j+1/2,k}^{n}-V_{i,j-1/2,k}^{n}}{\Delta y},\nonumber\\
\delta_{z}V_{i,j,k}^{n}&=&\frac{V_{i,j,k+1/2}^{n}-V_{i,j,k-1/2}^{n}}{\Delta z}.\nonumber
\end{eqnarray}

Then, we obtain that,
\begin{eqnarray}
U_{i,j,k}^{n}&=&u_{i,j,k}^{0}+\Delta t\sum_{m=1}^{s}a_{nm}\Big(-\Sigma U_{i,j,k}^{m}+\widetilde{P(\nabla)}U_{i,j,k}^{m}\Big),\label{4.47}\\
u_{i,j,k}^{1}&=&u_{i,j,k}^{0}+\Delta t\sum_{m=1}^{s}b_{m}\Big(-\Sigma U_{i,j,k}^{m}+\widetilde{P(\nabla)}U_{i,j,k}^{m}\Big),\label{4.48}
\end{eqnarray}
with
$$\widetilde{P(\nabla)}=\left(\begin{array}{ccccccc}
&&&&\delta y&\delta y&\\
&&-\delta z&-\delta z&&&\\
\delta z&\delta z&&&&&\\
&&&&-\delta x&-\delta x&\\
&&\delta x&\delta x&&&\\
-\delta y&-\delta y&&&&&\\
\end{array}\right). $$

\begin{theorem}\label{them4.3}

For and integer $n\geq0$, set
$\textbf{E}_{i,j,k}^{n}=(E_{x_{i,j,k}}^{n}, E_{y_{i,j,k}}^{n},
E_{z_{i,j,k}}^{n})$ and $\textbf{H}_{i,j,k}^{n}=(H_{x_{i,j.k}}^{n},
H_{y_{i,j,k}}^{n}, H_{z_{i,j,k}}^{n})$ be the solution of the
discrete scheme (\ref{4.47})-(\ref{4.48}), then the discrete version of the
energy evolution law is,
\begin{small}
\begin{eqnarray}\label{4.49}
\|\textbf{E}^{n+1}\|_{h}^{2}+\|\textbf{H}^{n+1}\|_{h}^{2}&=&\|\textbf{E}^{n}\|_{h}^{2}+\|\textbf{H}^{n}\|_{h}^{2}\nonumber\\
&-&2\Delta t\sigma\sum_{m=1}^{s}b_{m}((\textbf{E}_{i,j,k}^{m})^{2}+(\textbf{H}_{i,j,k}^{m})^{2}),
\end{eqnarray}
\end{small}
where
\begin{eqnarray}
\|\textbf{E}^{n}\|_{h}^{2}&=&h_{x}h_{y}h_{z}\sum_{i=r_{1}}^{s}\sum_{j=r_{2}}^{s}\sum_{k=r_{3}}^{s}((E_{x_{i,j,k}}^{n})^{2}+(E_{y_{i,j,k}}^{n})^{2}+(E_{z_{i,j,k}}^{n})^{2}),\nonumber\\
\|\textbf{H}^{n}\|_{h}^{2}&=&h_{x}h_{y}h_{z}\sum_{i=r_{1}}^{s}\sum_{j=r_{2}}^{s}\sum_{k=r_{3}}^{s}((H_{x_{i,j,k}}^{n})^{2}+(H_{y_{i,j,k}}^{n})^{2}+(H_{z_{i,j,k}}^{n})^{2}).\nonumber
\end{eqnarray}
\end{theorem}

\begin{proof}
For simplicity, we introduce the following notation
\begin{equation}\label{4.50}
f_{i,j,k}^{m}=-\Sigma U_{i,j,k}^{m}+\widetilde{P(\nabla)}U_{i,j,k}^{m}.
\end{equation}

From the discrete scheme (\ref{4.48}), we can get that the following
relation between $u_{i,j,k}^{1}$ and $u_{i,j,k}^{0}$:
\begin{small}
\begin{eqnarray}\label{4.51}
(u_{i,j,k}^{1})^{T}u_{i,j,k}^{1}&=&(u_{i,j,k}^{0}+\Delta t\sum_{m=1}^{s}b_{m}f_{i,j,k}^{m})^{T}(u_{i,j,k}^{0}+\Delta t\sum_{m=1}^{s}b_{m}f_{i,j,k}^{m})\nonumber\\
&=&((u_{i,j,k}^{0})^{T}u_{i,j,k}^{0}+\Delta t\sum_{m=1}^{s}b_{m}(f_{i,j,k}^{m})^{T})(u_{i,j,k}^{0}+\Delta t\sum_{m=1}^{s}b_{m}f_{i,j,k}^{m})\nonumber\\
&=&(u_{i,j,k}^{0})^{T}u_{i,j}^{0}+\Delta t((u_{i,j,k}^{0})^{T}\sum_{m=1}^{s}b_{m}f_{i,j,k}^{m})\nonumber\\
&+&\Delta t(\sum_{m=1}^{s}b_{m}(f_{i,j,k}^{m})^{T}u_{i,j,k}^{0})
+(\Delta t)^{2}\sum_{m,n=1}^{s}b_{m}b_{n}(f_{i,j,k}^{m})^{T}f_{i,j,k}^{n}. \nonumber
\end{eqnarray}
\end{small}

Note that the discrete scheme (\ref{4.47}) and the notation (\ref{4.50}), it can
be rewritten as
\begin{equation}\label{4.52}
u_{i,j,k}^{0}=U_{i,j,k}^{n}-\Delta t\sum_{m=1}^{s}a_{nm}f_{i,j,k}^{m}.
\end{equation}

Then, we insert the relation between $u_{i,j,k}^{1}$ and $u_{i,j,k}^{0}$ into (\ref{4.52}), we obtain that
\begin{eqnarray}\label{4.53}
&&(u_{i,j,k}^{1})^{T}u_{i,j,k}^{1}\nonumber\\
&=&(u_{i,j,k}^{0})^{T}u_{i,j,k}^{0}+\Delta t\sum_{m=1}^{s}b_{m}((U_{i,j,k}^{n})^{T}f_{i,j}^{m}+(f_{i,j,k}^{m})^{T}U_{i,j,k}^{n})\nonumber\\
&+&(\Delta t)^{2}\sum_{m,n=1}^{s}(b_{m}b_{n}-b_{m}a_{mn}-b_{n}a_{nm})(f_{i,j,k}^{m})^{T}f_{i,j,k}^{n}.
\end{eqnarray}

Due to the symplectic condition (\ref{4.46}), we have
\begin{eqnarray}\label{4.54}
(u_{i,j,k}^{1})^{T}u_{i,j,k}^{1}=(u_{i,j,k}^{0})^{T}u_{i,j,k}^{0}&+&\Delta t\sum_{m=1}^{s}b_{m}((U_{i,j,k}^{n})^{T}f_{i,j,k}^{m}\nonumber\\
&+&(f_{i,j,k}^{m})^{T}U_{i,j,k}^{n}).
\end{eqnarray}

Since the Maxwell equations energy conserving law in lossless medium, we obtain
$$(U_{i,j,k}^{m})^{T}(\widetilde{P(\nabla)}U_{i,j,k}^{m})=0.$$

Therefore, (\ref{4.54}) becomes the following form
\begin{small}
\begin{equation}\label{4.55}
(u_{i,j,k}^{1})^{T}u_{i,j,k}^{1}=(u_{i,j,k}^{0})^{T}u_{i,j,k}^{0}-2\Delta t\sigma\sum_{m=1}^{s}(b_{m}(U_{i,j,k}^{m})^{T}\Sigma U_{i,j,k}^{m}).
\end{equation}
\end{small}

From the presentation of $\Sigma$ and $U$, we have
\begin{small}
\begin{equation}\label{4.56}
(u_{i,j,k}^{1})^{T}u_{i,j,k}^{1}=(u_{i,j,k}^{0})^{T}u_{i,j,k}^{0}-2\Delta t\sigma\sum_{m=1}^{s}b_{m}((\textbf{E}_{i,j,k}^{m})^{2}+(\textbf{H}_{i,j,k}^{m})^{2}).
\end{equation}
\end{small}

Summing all terms in the above equation (\ref{4.56}) over all spatial
indices $i$, $j$, $k$, then we can get the energy evolution law
(\ref{4.49}). The proof is finished.
\end{proof}

\nocite{*}
\bibliography{aipsamp}

\end{document}